\documentclass[12pt]{amsart}
\usepackage{amssymb,latexsym,amsfonts}
\usepackage{amsmath}
\usepackage{amscd}
\usepackage{pb-diagram,pb-xy}%lamsarrow,pb-lams}

\usepackage{graphicx}
\usepackage{hyperref}
\usepackage{color}
\usepackage[matrix,arrow]{xy}
\newcommand{\incircle}{\raisebox{2pt}{\tiny {\textcircled{\raisebox{-0.5pt} {{\tiny r}}}}}}

\newcommand{\iincircle}{\raisebox{2pt}{\tiny {\textcircled{\raisebox{-0.5pt} {{\tiny i}}}}}}

\newcommand{\ccc}{{ c}\,}
\newcommand{\GG}{{ g}\,}
\newcommand{\cc}{{\rm c}}

\DeclareMathAlphabet{\mathpzc}{OT1}{pzc}{m}{it}

\newtheorem*{prop-non}{Proposition}
\newtheorem*{thm-non}{Theorem}
\newtheorem{theorem}{Theorem}[section]
\newtheorem{corollary}[theorem]{Corollary}
\newtheorem{lemma}[theorem]{Lemma}
\newtheorem{remark}{Remark}[section]
\newtheorem{proposition}[theorem]{Proposition}
\newtheorem{thm}{Theorem}[section]
\theoremstyle{definition}
\newtheorem{definition}[thm]{Definition}

\begin{document}

 \author{A. Belkhirat and M. L. Labbi}
 \title[Composition and exterior products]{ Interactions between the composition and  exterior products of double forms and applications}
   \date{}
\subjclass[2010]{Primary 53B20, 53C21; Secondary  15A24, 15A69.}
\keywords{Composition product of double forms, exterior product of double forms, interior product of a double form, mixed exterior algebra, Pontrjagin form, Pontrjagin number, pure curvature tensor. }
\thanks{}
\begin{abstract} We translate into the double forms formalism the basic identities of Greub and Greub-Vanstone that were obtained  in the mixed exterior algebra. In particular, we  introduce a second product in the space of double forms, namely the composition product, which provides this space with a second associative algebra structure.  The composition product interacts with the exterior product of double forms;  the resulting relations provide simple  alternative proofs to  some classical   linear algebra identities as well as  to recent results in the exterior algebra of double forms.\\
We define a refinement of the notion of pure curvature of  Maillot and  we use one of the basic identities to prove that if a Riemannian $n$-manifold has $k$-pure curvature  and $n\geq 4k$  then its Pontrjagin class of degree $4k$ vanishes. 
\end{abstract}
   \maketitle

\tableofcontents

\section{Introduction}
Let $h$ be an endomorphism (or a bilinear form) of an Euclidean real vector space $(V,g)$  of dimension $n<\infty$. Recall the classical Girard-Newton identities for $1\leq r\leq n$
\begin{equation*}
rs_r(h)=\sum_{i=1}^r(-1)^{i+1}s_{r-i}(h)p_ i(h).
\end{equation*}
Where $p_i(h)$ is the trace of the endomorphism $h^{ \iincircle}=\underbrace{h\circ ... \circ h}_\text{$i$-times},$ and $s_i(h)$ are the symmetric functions in the (possibly complex) eigenvalues of $h$. It turns out that the invariants $s_i(h)$ are also traces of endomorphisms constructed from $h$ and the metric $g$ using exterior product of double forms \cite{Labbialgebraic}.\\
Another celebrated classical result which illustrates also the interaction between the composition and exterior product is the Cayley-Hamilton theorem
\begin{equation*}
\sum_{r=0}^n(-1)^rs_{n-r}(h)h^{ \incircle}=0.
\end{equation*}
The first identity is a scalar valued identity, the second one is a endomorphism (or bilinear form) valued identity. Higher double forms valued identities where obtained in \cite{Labbialgebraic}. In particular, it is shown that the infinitisimal version of the general Gauss-Bonnet theorem is a double forms valued identity of Cayley-Hamilton type which again involves the two products. Another illustration of the importance of these interactions is the  expression of all Pontrjagin numbers of a compact oriented manifold  of dimension $n=4k$ as the integral
of the following $4k$-form \cite{Labbi-Bianchi}
\begin{equation*}
 P_1^{k_1}P_2^{k_2}\cdots P_m^{k_m}=\frac{(4k)!}{[(2k)!]^2(2\pi)^{2k}}  \Big(\prod_{i=1}^{m}\frac{[(2i)!]^2}{(i!)^{2k_i}(4i)!}\Big){\rm Alt}\Big[(R\circ R)^{k_1}(R^2\circ R^2)^{k_2}\cdots (R^m\circ R^m)^{k_m}\Big].
\end{equation*}
Where $R$ is the  Riemann curvature tensor  seen as a $(2,2)$ double form,  $ k_1, k_2, ..., k_m$   are non-negative integers  such that $k_1+2k_2+...+mk_m =k$, ${\rm Alt}$ is the alternating operator, and where all the powers over double forms are taken with respect to the exterior product of double forms.\\
In this study, we investigate the interactions between these two products. The paper is organized as follows. Sections 2 and 3, are about definitions and  basic facts about the exterior and composition products of double forms. In section 4, we introduce and study the interior product of double forms which generalizes the usual Ricci contactions. Precisely, for a double form $\omega$, the interior product map $i_\omega$, which maps a double form to another double form,  is the adjoint of the exterior  multiplication map by $\omega$. In particular, if $\omega=g$ we recover the usual Ricci contraction map of double forms.\\
Section 5 is about some natural extensions of endomorphisms of $V$ onto endomorphisms of the exterior algebra of double forms. We start with an endomorphism $h:V\rightarrow V$, there exists a unique exterior algebra endomorphism $\widehat{h}:\Lambda V\rightarrow \Lambda V$ that extends $h$ and such that $\widehat{h}(1)=1.$   Next,  the space  $\Lambda V\otimes \Lambda V$ can be regarded in two ways as a $\Lambda V$-valued exterior vectors,  therefore the endomorphism  $\widehat{h}$ operates on the space $\Lambda V\otimes \Lambda V$   in two natutal ways, say $\widehat{h}_R$ and $\widehat{h}_L$. The so obtained two endomorphisms are in fact  exterior  algebra endomorphisms. We prove that  the endomorphisms $\widehat{h}_R$  and $\widehat{h}_L$  are nothing but the   right and left multiplication maps in the composition algebra, precisely we prove that
\begin{equation*}
\widehat{h}_R(\omega)=e^h\circ \omega,\, {\rm and}\,\,  \widehat{h}_L(\omega)=\omega\circ e^{(h^t)}.
\end{equation*}
Where $e^h:=1+h+\frac{h^2}{2!}+\frac{h^3}{3!}+...$ and the powers are taken with respect to the exterior product of double forms. As a consequence of this discussion we get easy proofs of classical linear algebra including Laplace expansions of the determinant.\\
In section 6,  we first  state and prove  Greub's basic identity  relating the exterior and composition  products of double forms:
 \begin{prop-non}
 If $h, h_1,...,h_p$ are bilinear forms on $V$, and $h_1...h_p$  is their exterior product then
 \begin{equation*}
\begin{split}
 i_h(h_1...h_p)&=\sum_{j}\langle h,h_j\rangle  h_1...\hat{h}_j...h_p\\
 &-\sum_{j<k}(h_j\circ h^t\circ h_k+h_k\circ h^t\circ h_j)h_1...\hat{h}_j...\hat{h}_k...h_p.
\end{split}
 \end{equation*}
 \end{prop-non}
 Consequently, for a bilinear form $k$ on $V$ ,  the contraction of $\cc k^p$  of the exterior power $k^p$ of $k$ is given by 
 \begin{equation*}
 \cc k^p=p(\cc k) k^{p-1}-p(p-1)(k \circ k)k^{p-2}.
 \end{equation*}
Using the fact that the diagonal sub-algebra (the subspace of all $(p,p)$ double forms, $p\geq 0$) is spanned by exterior products of bilinear forms on $V$, we obtain the following useful  formula as a  consequence of the previous identity. This new formula generalizes formula (15) of \cite{Labbidoubleforms} in Theorem 4.1 to double forms that are not symmetric or do not satisfy the first Bianchi identity
\begin{equation*}
\ast \Big( \frac{g^{k-p}\omega}{(k-p)!}\Big)=\sum_{r}^{}(-1)^{r+p}\frac{g^{n-p-k+r}}{(n-p-k+r)!}\frac{c^r}{r!}(\omega^t).
\end{equation*}
Where $\ast$ is the double Hodge  star operator on double forms. \\
In the same section 6, we state and prove another identity relating the exterior and composition product of double forms, namely  the following Greub-Vanstone basic identity 
\begin{thm-non}
For $1\leq p\leq n$, and for bilinear forms $h_1,...,h_p$ and $k_1,...,k_p$ we have
\[(h_1h_2...h_p)\circ (k_1k_2...k_p)=\sum_{\sigma\in S_p}(h_1\circ k_{\sigma(1)})...(h_p\circ k_{\sigma (p)})=\sum_{\sigma\in S_p}(h_{\sigma (1)}\circ k_{1})...(h_{\sigma (p)}\circ k_{p})\]
In particular, when $h=h_1=...=h_p$ and $k=k_1=...=k_p$ we have the following nice relation:
\begin{equation*}
h^p\circ k^p=p!(h\circ k)^p.
\end{equation*}
\end{thm-non}
The last section 7,  is devoted to the study of $p$-pure Riemannian manifolds. Let $1\leq p\leq n/2$ be a positive integer, a Riemannian $n$-manifold  is said to have a \emph{$p$-pure curvature tensor} if at each point of the manifold    the curvature operator that is associated to the exterior power  $R^p$ of  the Riemann curvature tensor $R$  has decomposed eigenvectors. For $p=1$, we recover the usual pure Riemannian manifolds of Maillot. A pure manifold is always $p$-pure for $p\geq 1$, we give examples of $p$-pure Riemannian manifolds that are $p$-pure for some $p>1$ without being pure. The main result of this section is the following
\begin{thm-non}
If a Riemannian $n$-manifold is $k$-pure and $n\geq 4k$  then its Pontrjagin class of degree $4k$ vanishes. 
\end{thm-non}
The previous theorem refines a result by Maillot  in \cite{Maillot}, where he proved  that all pontrjagin classes of a pure Riemannian manifold vanish. \\
Finally, we emphasize that sections 3,4,5 and 6 are mainly a translation into the language of double forms of some works by Greub \cite{Greub-book} and Greub-Vanstone \cite{Greub-Vanstone} in the context of Mixed exterior algebra. We hope that by this contribution, we shed light on  these important  contributions of Greub and Vanstone which  are  not well  known whithin the geometers community.

\section{ The Exterior Algebra of Double Forms}\label{prem.}

Let $(V,g)$ be an Euclidean real vector space  of finite dimension $n$.  In the
 following we shall identify whenever convenient (via their Euclidean structures),
the vector spaces
 with their duals. Let
  $\Lambda V^{*}=\bigoplus_{p\geq 0}\Lambda^{p}V^{*}$  (resp.
   $\Lambda V=\bigoplus_{p\geq 0}\Lambda^{p}V$) denotes the exterior algebra
 of the dual space  $V^* $ (resp.   $V$). Considering  tensor products,
  we define the space of double exterior forms of $V$ (resp. double exterior vectors)  as
 $${\mathcal D}(V^*)= \Lambda V^{*}\otimes \Lambda V^{*}=\bigoplus_{p,q\geq 0}
  {\mathcal D}^{p,q}(V^*),$$
$${\rm resp.}\,\, {\mathcal D}(V)= \Lambda V\otimes \Lambda V=\bigoplus_{p,q\geq 0}
  {\mathcal D}^{p,q}(V),$$
 where $  {\mathcal D}^{p,q}(V^*)= \Lambda^{p}V^{*} \otimes  \Lambda^{q}V^{*},$ resp. $  {\mathcal D}^{p,q}(V)= \Lambda^{p}V \otimes  \Lambda^{q}V.$
The space ${\mathcal D}(V^*)$ is naturally  a bi-graded associative  algebra, called \emph{double exterior algebra of $V$},
 where for $\omega_1=\theta_1\otimes \theta_2\in { \mathcal D}^{p,q}(V^*)$ and
 $\omega_2=\theta_3\otimes \theta_4\in  {\mathcal D}^{r,s}(V^*)$, the multiplication is given  by
 \begin{equation}
 \label{def:prod}
 \omega_1\omega_2= (\theta_1\otimes \theta_2 )(\theta_3\otimes
 \theta_4)=
 (\theta_1\wedge \theta_3 )\otimes(\theta_2\wedge \theta_4)\in
    {\mathcal D}^{p+r,q+s}(V).\end{equation}
Where $\wedge$ denotes the standard exterior product on the exterior algebra $ \Lambda V^{*}$. The product in the exterior algebra of double vectors is defined in the same way.\\
A \emph{double exterior  form of degree $(p,q)$}, (resp.  a \emph{double exterior vector of degree $(p,q)$}) is by definition an  element of the tensor product
    $  {\mathcal D}^{p,q}(V^*)= \Lambda^{p}V^* \otimes \Lambda^{q}V^*$, (resp.    $  {\mathcal D}^{p,q}(V)= \Lambda^{p}V\otimes \Lambda^{q}V$.
     It can be identified canonically
     with a bilinear form $\Lambda^pV\times\Lambda^qV\rightarrow {\bf R}$, which in turn can be seen as
     a multilinear form which is skew symmetric in the first $p$-arguments and also
     in the last $q$-arguments.\\
The above multiplication in ${\mathcal D}(V^*)$  (resp. ${\mathcal D}(V)$) shall be  called the \emph{exterior product of double forms,} (resp. \emph{exterior product of double vectors.})\\

Recall that the (Ricci) contraction map, denoted by $c$,  maps  ${\mathcal D}^{p,q}(V^*)$ into  ${\mathcal D}^{p-1,q-1}(V^*)$. For a double form  $\omega \in  {\mathcal D}^{p,q}(V^*)$  with $p\geq 1$ and $q\geq 1$, we  have
$$\ccc \omega(x_1\wedge...\wedge x_{p-1},y_1\wedge...\wedge y_{q-1})=
\sum_{j=1}^{n}\omega(e_j\wedge x_1\wedge... x_{p-1},
e_j\wedge y_1\wedge...\wedge y_{q-1})$$
where $\left\{e_1,...,e_n\right\}$ is an arbitrary  orthonormal basis of $V$ and $\omega$ is seen  as a bilinear form as explained above.If $p=0$ or $q=0$, we set $\ccc \omega=0$.\\
It turns out, see \cite{Labbidoubleforms}, that the contraction map $\ccc$ on ${\mathcal D(V^*)}$  is  the adjoint of  the multiplication map by the metric  $g$ of $V$,  precisely we have for $\omega_1, \omega_2 \in  {\mathcal D}(V^*)$ the following
\begin{equation}\label{adj:gc}
<\GG \omega_1,\omega_2>=<\omega_1,\ccc \omega_2>.
\end{equation}
Suppose now that  we have fixed an orientation on the vector space  $V$.
The classical Hodge star operator  $*:\Lambda^{p}V^*\rightarrow \Lambda^{n-p}V^*$ can be extended naturally to operate
on double forms as follows. For a $(p,q)$-double form $\omega$ (seen as a bilinear form), $*\omega$ is  the  $(n-p,n-q)$-double form  given by
\begin{equation}\label{remark1}
 *\omega(.,.)=(-1)^{(p+q)(n-p-q)}\omega(*.,*.).
\end{equation}
Note that $*\omega$ does not depend on the chosen orientation as the usual Hodge star operator is applied twice.
The so-obtained  operator is still called the Hodge star operator operating on double forms or the \emph{double Hodge star operator}. This new operator  provides another simple  relation between the contraction map $c$ of double forms
and the multiplication map by the metric as follows:
\begin{equation}
\label{A}
\GG \omega=*\,\ccc*\omega.
\end{equation}
Furthermore, the double Hodge star operator generates the inner product of double forms  as follows.  For any two double forms $\omega,\theta\in {\mathcal D}^{p,q}$ we have
\begin{equation}\label{remark2}
<\omega,\theta>=*\Bigl(  \omega(*\theta)\Bigr)=(-1)^{(p+q)(n-p-q)}*\Bigl((*\omega)\theta\Bigr).
\end{equation}
The reader is kindly  invited to consult the  proofs  of the above  relations in  \cite{Labbidoubleforms}.
%\begin{remark}
%We bring to the reader's attention that formulas  (\ref{remark1}) and (\ref{remark2})   correct a sign omission in  %(\cite{Labbidoubleforms}), precisely,  in  relation (8) and in  the second formula of  section (3.2)  of  (\cite{Labbidoubleforms}).
 %We note that this sign omission does not affect any of the results of \cite{Labbidoubleforms}, in fact .... Note that this was %already corrected in subsequent papers of the second author, see for instance formula .. in ...
%\end{remark}

\begin{definition}
The subspace
$$\Delta V^*=\bigoplus_{p\geq 0}
  {\mathcal D}^{p,p}(V^*),\, ({\rm resp.}\,\, \Delta V=\bigoplus_{p\geq 0}
  {\mathcal D}^{p,p}(V))$$
of $\mathcal{D}(V^*)$ (resp. $\mathcal{D}(V))$ is a commutative subalgebra and shall be called the \emph{ diagonal subalgebra}.
\end{definition}

\section{The composition  Algebra of Double Forms}
The space ${\mathcal D}= \Lambda V^{*}\otimes \Lambda V^{*}$ is canonically isomorphic to the space of linear endomorphisms $L( \Lambda V,\Lambda V)$. Explicitely, we have  the following canonical isomorphism
\begin{equation}
\begin{split}
 \mathcal{T}: \Lambda V^{*}\otimes \Lambda V^{*}&\rightarrow  L( \Lambda V,\Lambda V)\\
\omega_1\otimes \omega_2 &\rightarrow   \mathcal{T}(\omega_1\otimes \omega_2)
\end{split}
\end{equation}
is given by
\[   \mathcal{T}(\omega_1\otimes \omega_2)(\theta)=\langle \omega_1^\sharp,\theta\rangle \omega_2^\sharp.\]
 Where   $ \omega_i^\sharp$  denotes the  exterior vector dual to the exterior form  $\omega_i$.\\
Note that if we look at a double form $\omega$  as a bilinear form on $ \Lambda V$, then  $ \mathcal{T}(\omega)$ is no thing but the canonical linear operator asociated to the bilinear form $\omega$.\\

It is easy to see that $ \mathcal{T}$ maps for each $p\geq 1$ the double form $\frac{g^p}{p!}$ to the identity map in $L( \Lambda^p V,\Lambda^p V)$, in particular $ \mathcal{T}$ maps the double form $1+g+\frac{g^2}{2!}+ ...$ onto the identity map in $L( \Lambda V,\Lambda V)$.\\
The space $L( \Lambda V,\Lambda V)$ is an algebra under the composition product $\circ$  that is not isomorphic to the algebra of double forms. Pulling back the operation $\circ$ to  ${\mathcal D}$ we obtain a second multiplication in  ${\mathcal D}$ which we shall call \emph{the composition product of double forms} or  \emph{Greub's product of double forms} and will be still denoted by $\circ$.\\
More explicitly, given  two simple double forms
 $\omega_1=\theta_1\otimes \theta_2\in { \mathcal D}^{p,q}$ and
    $\omega_2=\theta_3\otimes \theta_4\in  {\mathcal D}^{r,s}$, we have
    \begin{equation}
\omega_1\circ\omega_2=(\theta_1\otimes \theta_2)\circ (\theta_3\otimes \theta_4)=\langle \theta_1,\theta_4\rangle \theta_3\otimes \theta_2\in  {\mathcal D}^{r,q}.
\end{equation}
It is clear that $\omega_1\circ\omega_2=0$ unless $p=s$.\\
Alternatively, if we look at $\omega_1$ and $\omega_2$ as bilinear forms, then the composition product read \cite{Labbialgebraic}

\begin{equation}
\omega_1\circ\omega_2(u_1,u_2)
=\sum_{\scriptstyle i_1<i_2<...<i_p}\omega_2(u_1,e_{i_1}\wedge ...\wedge e_{i_p})\omega_1(e_{i_1}\wedge ...\wedge e_{i_p},u_2).
\end{equation}
Where $\{e_1,...,e_n\}$ is an arbitrary orthonormal basis of $(V,g)$,  $u_1\in \Lambda^r$  is an $r$-vector and $u_2\in \Lambda^q$ is a $q$-vector in $V$.

We list below some properties of this product.
\subsubsection{Transposition of double forms}
For a double form $\omega \in  {\mathcal D}^{p,q}$, we denote by $\omega^t\in  {\mathcal D}^{q,p}$ the transpose of $\omega$, which is defined by
\begin{equation}
\omega^t(u_1,u_2)=\omega(u_2,u_1).
\end{equation}
Alternatively, if $\omega=\theta_1\otimes \theta_2$ then

\begin{equation}
\omega^t=(\theta_1\otimes \theta_2)^t=\theta_2\otimes \theta_1.
\end{equation}

A double form $\omega$ is said to be \emph{ a symmetric double form}  if  $\omega^t=\omega$.
\begin{proposition}\label{transpose-properties}
Let $\omega_1$, $\omega_2$ be two arbitrary  elements of ${\mathcal D}$, then
\begin{enumerate}
\item  $(\omega_1\circ \omega_2)^t=\omega_2^t\circ \omega_1^t$ and  $(\omega_1 \omega_2)^t=\omega_1^t\omega_2^t$.
\item $ \mathcal{T}(\omega_1^t)=( \mathcal{T}(\omega_1))^t$.
\item If $\omega_3$ is a third double form then $\langle \omega_1\circ\omega_2,\omega_3\rangle=\langle \omega_2,\omega_1^t\circ\omega_3\rangle=\langle \omega_1,\omega_3\circ \omega_2^t\rangle.$

\end{enumerate}
\end{proposition}
\begin{proof}
%\begin{enumerate}
Without loss of generality, we may assume that
$\omega_1=\theta_1\otimes\theta_2$ and $\omega_2=\theta_3\otimes\theta_4$
then,
\begin{equation*}
\begin{split}
 (\omega_1\circ \omega_2)^t &=\Bigl((\theta_1\otimes \theta_2)\circ (\theta_3\otimes \theta_4)\Bigr)^t = <\theta_1,\theta_4>(\theta_3\otimes\theta_2)^t=<\theta_1,\theta_4>(\theta_2\otimes\theta_3)\\
 &=(\theta_4\otimes \theta_3)\circ (\theta_2\otimes\theta_1)=(\theta_3\otimes\theta_4)^t\circ(\theta_1\otimes\theta_2)^t=\omega_2^t\circ\omega_1^t.
\end{split}
\end{equation*}
Similarly.
\[ (\omega_1 \omega_2)^t=\theta_2\wedge\theta_4\otimes \theta_1\wedge\theta_3=\omega_1^t\omega_2^t.\]
This proves $(1)$. Next, we prove prove relation $(2)$ as follows, 

\begin{equation*}
\begin{split}
 < \mathcal{T} (\omega_1^t)(u_1),u_2>&=< \mathcal{T} ((\theta_1\otimes \theta_2)^t)(u_1),u_2>=< \mathcal{T} (\theta_2\otimes \theta_1)(u_1),u_2>\\
&=<\theta_2^\sharp ,u_1>    <\theta_1^\sharp, u_2>=<u_1, <\theta_1^\sharp,u_2>\theta_2^\sharp>=<u_1,  \mathcal{T}(\theta_1\otimes \theta_2)u_2>\\
&=\langle \Bigl( \mathcal{T}(\theta_1\otimes \theta_2)\Bigr)^t(u_1),u_2\rangle=\langle \Bigl( \mathcal{T}(\omega_1)\Bigr)^t(u_1),u_2\rangle.
\end{split}
\end{equation*}
Finally we prove (3). without loss of  generality assume as above that the three double forms are simple, let $\omega_3=\theta_5\otimes\theta_6$ then a simple computation shows that
\begin{align*}
\langle \omega_1\circ\omega_2,\omega_3\rangle &=\langle \theta_1,\theta_4\rangle \langle \theta_3\otimes  \theta_2,\theta_5\otimes \theta_6\rangle=\langle \theta_1,\theta_4\rangle \langle \theta_3,\theta_5\rangle \langle \theta_2,\theta_6\rangle .\\
\langle \omega_2,\omega_1^t\circ\omega_3\rangle &=\langle \theta_3\otimes\theta_4,\langle \theta_2,\theta_6\rangle  \theta_5\otimes \theta_1=\langle \theta_1,\theta_4\rangle \langle \theta_3,\theta_5\rangle \langle \theta_2,\theta_6\rangle .\\
\langle \omega_1,\omega_3\circ \omega_2^t\rangle  &=\langle \theta_1\otimes\theta_2,\langle \theta_5,\theta_3\rangle  \theta_4\otimes \theta_6=\langle \theta_1,\theta_4\rangle \langle \theta_3,\theta_5\rangle \langle \theta_2,\theta_6\rangle .\\
\end{align*}
This completes the proof of the proposition.
\end{proof}
The composition product provides another  useful formula for the inner product of double forms as follows
\begin{proposition}[\cite{Labbialgebraic}]
The inner product of two double forms $\omega_1,\omega_2\in  {\mathcal D}^{p,q}$ is the full contraction of the composition  product $\omega_1^t\circ\omega_2$ or $\omega_2^t\circ\omega_1$ , precisely we have
\begin{equation}
\langle \omega_1,\omega_2\rangle=\frac{1}{p!}\cc^p(\omega_2^t\circ\omega_1)=\frac{1}{p!}\cc^p(\omega_1^t\circ\omega_2.)
\end{equation}
\end{proposition}
\begin{proof}
We use the fact that the contraction map $\cc$ is the adjoint of the exterior multiplication map by $g$ and the above proposition  as follows
\[\frac{1}{p!}\cc^p(\omega_2^t\circ\omega_1)=\langle \omega_2^t\circ\omega_1,\frac{g^p}{p!}\rangle =\langle \omega_1,\big( \omega_2^t\big)^t\circ \frac{g^p}{p!}=\langle \omega_1,\omega_2\rangle.\]
Where we used the fact that $\frac{g^p}{p!}$ is a unit element in the composition algebra. The prove of the second relation is similar.
\end{proof}
\begin{remark}\label{remark1}
The inner product used by Greub and Vanstone in \cite{Greub-book, Vanstone, Greub-Vanstone} is the pairing product which can be defined by
\[\langle\langle  \omega_1,\omega_2\rangle\rangle=\frac{1}{p!}\cc^p(\omega_2\circ\omega_1)=\frac{1}{p!}\cc^p(\omega_1\circ\omega_2.)\]
This is clearly different from the inner product that we are using in this paper. The two products coincide if $\omega_1$ or $\omega_2$ is a symmetric double form.
\end{remark}

\section{Interior  product for double forms}
Recall that for a vector $v\in V$,  the interior  product map $i_v:  \Lambda^p V^{*}\rightarrow \Lambda^{p-1} V^{*}$, for $p\geq 1$,   is defined by declaring
$$i_v\alpha(x_2,...,x_p)=\alpha(v,x_2,...,x_p).$$
There are two natural ways to extend this operation to double forms seen as bilinear maps as above,  Precisely we define the inner product map $i_v: {\mathcal D}^{p,q}\rightarrow {\mathcal D}^{p-1,q}$  for $p\geq 1$, and the adjoint   inner product map
 $\tilde{i}_v:   {\mathcal D}^{p,q}\rightarrow {\mathcal D}^{p,q-1}$ for $q\geq 1$ by declaring

$$i_v\omega (x_2,...,x_p;y_1,...,y_q))=\omega(v,x_2,...,x_p;y_1,...,y_q),$$
and
$$\tilde{i}_v(\omega)(x_1,...,x_p;y_2,...,y_q))=\omega(x_1,...,x_p;v,y_2,...,y_q).$$

Note that the first map are nothing but the usual interior product  of vector valued $p$-forms. The second map can be obtained from the first one via transposition as follows
$$\tilde{i}_v(\omega)= \Big( i_v(\omega^t)\Big)^t.$$
 In particular, the maps $i_v$ and  $\tilde{i}_v$  satisfy the same  algebraic properties as  the usual interior product of usual forms.  \\

Next,  we define a new  natural (diagonal) interior product on double forms as follows. Let $v\otimes w\in V\otimes V$  be a decomposable $(1,1)$  double vector, we define  $i_{v\otimes w}: {\mathcal D}^{p,q}\rightarrow {\mathcal D}^{p-1,q-1}$  for $p,q\geq 1$ by
$$i_{v\otimes w}=i_v\circ \tilde{i}_w.$$
Equivalently,
$$i_{v\otimes w}\omega (x_2,...,x_p;y_2,...,y_q))=\omega(v,x_2,...,x_p;w,y_2,...,y_q).$$
The previous map is obviously bilinear with respect to $v$ and $w$ and therefore can be extended and defined for  any $(1,1)$  double vector in $V\otimes V$.\\
Let $h$ be a  $(1,1)$ double form, that is a  bilinear form on $V$. Then in a basis of $V$ we have $h=\sum_ih(e_i,e_j)e_i^*\otimes e_j^*$. The dual $(1,1)$  double vector associated to $h$ via the metric $g$  denoted by $h^\sharp$, is  by definition
$$h^\sharp=\sum_ih(e_i,e_j)e_i\otimes e_j.$$
We then define the interior product  $i_h$ to be the interior product $i_{h^\sharp }$.
\begin{proposition}\label{iR}
Let $h$ be an arbitrary  $(1,1)$  double form, then
\begin{enumerate}
\item for any  $(1,1)$  double form $k$ we have
$$i_{h}k=i_{k}h=\langle h,k\rangle.$$
\item For any  $(2,2)$  double form $R$ we have
$$i_hR=\overset{\circ}{R}h.$$
Where for a  $(1,1)$  double form $h$,   $\overset{\circ}{R}h$ denotes the operator defined for instance in \cite{Besse-book}, by $$\overset{\circ}{R}h(a,b)=\sum_{i,j}h(e_i,e_j)R(e_i,a;e_j,b).$$
\item The exterior multiplication map by $h$ in  ${\mathcal D}(V^*)$ is the adjoint of the interior product map $i_h$, that is
$$\langle i_h \omega_1,\omega_2\rangle=\langle \omega_1,h\omega_2\rangle.$$
\item For $h=g$, we have  $i_g=c$ is the contraction map in  ${\mathcal D}(V^*)$ as  defined in the introduction.
\end{enumerate}
\end{proposition}

\begin{proof} To prove the first assertion, assume that $h=\sum_{i,j}h(e_i,e_j)e_i^*\otimes e_j^*$ and $k=\sum_{r,s}k(e_r,e_s)e_r^*\otimes e_s^*$, where $(e_i^*)$ is an orthonormal basis of $V^*$. Then
\begin{align*}
i_{h}k&=\sum_{i,j,r,s}h(e_i,e_j)k(e_r,e_s)i_{e_i\otimes e_j}(e_r^*\otimes e_s^*)=\sum_{i,j,r,s}h(e_i,e_j)k(e_r,e_s)\langle e_i,e_r\rangle\langle e_j,e_s\rangle\\
&=\sum_{i,j}h(e_i,e_j)k(e_i,e_j)=\langle h,k\rangle.
\end{align*}
Next, we have
\begin{equation*}
i_hR(a,b)=\sum_{i,j}h(e_i,e_j)i_{e_i\otimes e_j}R(a,b)=\sum_{i,j}h(e_i,e_j)R(e_i,a;e_j,b)=\overset{\circ}{R}h(a,b).
\end{equation*}
This proves statement 2. To prove the third one, assume without loss of generality that $h=v^*\otimes w^*$ is decomposed, then
%\begin{equation}
\begin{align*}
\langle i_h(\omega_1),\omega_2\rangle  &=\langle i_v \circ \tilde{i}_w(\omega_1),\omega_2\rangle
 =\langle  \tilde{i}_w(\omega_1),(v^*\otimes1)\omega_2\rangle\\
 &=\langle  \omega_1,(1\otimes w^*)(v^*\otimes1)\omega_2\rangle
=\langle \omega_1,(v^*\otimes w^*)\omega_2\rangle\\
&=\langle \omega_1,h\omega_2\rangle.
\end{align*}
%\end{equation}
To prove the last relation 4, let $(e_i^*)$ be an orthonormal basis of $V^*$ then  $g=\sum_{i=1}^ne_i^*\otimes e_i^*$ and 
\begin{align*}i_g\omega(x_1,...,x_{p-1}&;y_1,...,y_{q-1})=\sum_{i=1}^ni_{e_i}\circ \tilde{i}_{e_i}\omega(x_1,...,x_{p-1};y_1,...,y_{q-1})\\
&= \sum_{i=1}^n \omega(e_i,x_1,...,x_{p-1};e_i,y_1,...,y_{q-1})=\cc\omega(x_1,...,x_{p-1};y_1,...,y_{q-1}).
\end{align*}
This completes the proof of the proposition.
\end{proof}

More generally, for a fixed double form $\psi\in {\mathcal D}(V^*)$, following Greub we denote by $\mu_{\psi}: {\mathcal D}(V^*)\rightarrow {\mathcal D}(V^*)$ the left exterior multiplication map by $\psi$, precisely
$$\mu_{\psi}(\omega)=\psi\omega.$$
We then define the map $i_{\psi}: {\mathcal D}\rightarrow {\mathcal D}$ as the adjoint map of  $\mu$:
$$\langle i_{\psi}(\omega_1),\omega_2\rangle =\langle \omega_1, \mu_{\psi}(\omega_2)\rangle.$$
Note that part (3) of Proposition \ref{iR} shows that this  general interior product $ i_{\psi}$ coincides with the above one in case $\psi$ is a $(1,1)$ double form.\\

\begin{remark}\label{remark2}
Let us remark at this stage that the interior product of double forms defined here differs by a transposition from the inner product of Greub, this is due to the fact that he is using the pairing product as explained in remark  \ref{remark1}. Precisely, an interior product $i_\psi \omega$ in the sens of Greub will be equal to the interior product $i_{\psi^t}\omega$ as defined here in this paper.
\end{remark}

It is results directly from the definition that for any two double forms $\psi,\varphi$ we have
$$ \mu_{\psi}\circ\mu_{\varphi}=\mu(\psi\varphi).$$
Consequently, one immediately gets
\begin{equation}\label{i}
i_{\psi}\circ i_{\varphi}=i_{\varphi\psi}.
\end{equation}
Note that for $\omega \in {\mathcal D}^{p,q}$ and  $\psi \in  {\mathcal D}^{r,s}$ we have $i_\psi(\omega)\in  {\mathcal D}^{p-r,q-s}$ if  $p\geq r$ and $q\geq s$. Otherwise $i_\psi(\omega)=0$. Furthermore,  it results immediately from formula  (\ref{i}), and statement (4) of Proposition  (\ref{iR}) that:
\begin{equation}
 i_{g^k}(\omega)=\cc^k(\omega),
\end{equation}
for any $\omega\in \mathcal{D}$. Where $c$ is the contraction map, $c^k=\underbrace{c\circ ... \circ c}_\text{$k$-times}$ and $g^k$ is the exterior power of the metric $g$.\\

In particular, for $\omega=g^p$, we get $i_{g^k}g^p=c^k(g^p)$. Then a direct computation or by using the general formula in Lemma 2.1 in  \cite{Labbidoubleforms}, one gets  the following simple but useful identity 
\begin{proposition} For $1\leq k\leq p\leq n={\rm dim(V)}$ we have
 \begin{equation}\label{interiormetricproduct}
 i_{g^k}(\frac{g^p}{p!})=\frac{(n+k-p)!}{(p-k)!}\frac{g^{p-k}}{(n-p)!}.
\end{equation}
\end{proposition}

We now state and  prove  some other useful facts about the interior product of double forms.
\begin{proposition} Let $\omega\in {\mathcal D^{p,q}}(V^*) $, the double Hodge star  operator $*$  is related to the interior product via  the following relation
\begin{equation}\label{*i}
*\omega=i_{\omega}\frac{g^n}{n!}.
\end{equation}
More generally, for any integer $k$, such that $1\leq k\leq n$ we have 
\begin{equation}\label{*iw}
*\frac{g^{n-k}}{(n-k)!}\omega=i_{\omega}\frac{g^k}{k!}.
\end{equation}

\end{proposition}
\begin{proof}
Let $\omega\in {\mathcal D^{p,q}}(V^*)$  and $\theta\in {\mathcal D^{n-p,n-q}}(V^*)$ be  arbitrary double forms. To prove the previous proposition, it is sufficient to prove that
$$\langle*\omega,\theta\rangle=\langle i_{\omega}(\frac{g^n}{n!}),\theta\rangle.$$
Using Equation (\ref{def:innerpdct}), we have $\langle*\omega,\theta\rangle=(-1)^{(2n-p-q)(p+q-n)}*(*^2\omega\theta)=*(\omega\theta).$\\
Since $\omega\theta\in {\mathcal D^{n,n}}$  and ${\rm dim}({\mathcal D^{n,n}})=1$, then
$$\omega\theta=\langle \omega\theta,\frac{g^n}{n!}\rangle \frac{g^n}{n!}=\langle i_{\omega}(\frac{g^n}{n!}),\theta\rangle$$
This proves the first part of the proposition. The second part results from the first one and  equation (\ref{interiormetricproduct}) as follows 
$$*\frac{g^{n-k}}{(n-k)!}\omega=i_{\frac{g^{n-k}}{(n-k)!}\omega}(\frac{g^n}{n!})=i_{\omega}\circ i_{\frac{g^{n-k}}{(n-k)!}}\left(\frac{g^n}{n!}\right)=i_{\omega}\left(\frac{g^k}{k!}\right).$$

\end{proof}

\begin{proposition}
\begin{enumerate}
\item  For any two double forms  $\omega_1, \omega_2\in {\mathcal D}(V^*)$, we have
$$*(\omega_1\circ\omega_2)=*\omega_1\circ*\omega_2.$$
In other words, $*$ is a composition algebra endomorphism.
\item On the diagonal subalgebra $\Delta(V^*)$,  we have (formulas (11a) and (11b) in \cite{Greub-Vanstone})
$$*\circ \mu_\omega=i_\omega \circ *,\,{\rm and}\,\, \mu_\omega\circ *=*\circ i_\omega.$$
In particular, we get the relations
\begin{equation}\label{star-i-star}
*\mu_{\omega}*=i_\omega\, {\rm and}\,\, *i_{\omega}*=\mu_\omega.
\end{equation}
Where $\mu_\omega$ is the left exterior multiplication map by $\omega$ in $\Delta(V^*)$.

\end{enumerate}
\end{proposition}
\begin{proof}
%\begin{enumerate}
 To prove statement 1, we assume that $\omega_1, \omega_2\in {\mathcal D}(V)$ \\ $*(\omega_1\circ\omega_2)=*[(\theta_1\otimes\theta_2)\circ(\theta_3\otimes\theta_4)]=
    <\theta_1,\theta_4>*\theta_3\otimes*\theta_2$\\
    As $*$ is an isometry, we have:\\
    $$<\theta_1,\theta_4>*\theta_3\otimes*\theta_2=<*\theta_1,*\theta_4>*\theta_3\otimes*\theta_2
    =*\omega_1\circ*\omega_2.$$
To prove statement 2, let $\omega\in {\mathcal D}^{p,q}$ and $\varphi\in {\mathcal D}^{r,s}$, then
\begin{align*}
*\circ\mu_{\omega}(\varphi)&=*(\omega \varphi)=i_{\omega \varphi}(\frac{g^n}{n!})=(-1)^{pr+qs}i_{\varphi\omega}(\frac{g^n}{n!})\\
&=(-1)^{pr+qs}i_{\omega}\circ i_{\varphi}(\frac{g^n}{n!})=(-1)^{pr+qs}i_{\omega}\circ*\varphi.
\end{align*}
If $\omega, \varphi\in \Delta(V^*)$ then $p=q$ and $r=s$ and the result follows.
To prove the second  statement in (2) , just apply to the previous equation the double Hodge star operator twice, once from the left and once from the right, then use the fact that on the diagonal subalgebra we have      $*^2$  is the identity map.
\end{proof}

As a direct  consequence of the previous  formula (\ref{star-i-star}),  applied to  $\omega=g$,  we recover the following result,  Theorem 3.4   of \cite{Labbidoubleforms},
\[*\cc *= \mu_g \, {\rm and}\,\, *\mu_g*=\cc.\]

\section{Exterior extensions of the endomorphisms on $V$}

Let $h\in {\mathcal D}^{1,1}(V^*)$ be a $(1,1)$ double form on $V$, and let  $\bar h={\mathcal T}(h)$ be its associated endomorphism on $V$  via the metric $g$.\\
There exists a unique exterior algebra endomorphism  $\widehat{h}$ of $\Lambda V$ that extends  $\bar{h}$ and such that  $\widehat{h}(1)=1$ . Explicitely, for any set of vectors $v_1,...,v_p$ in $V$, the endomorphism is defined by declaring
$$\widehat{h}(v_1\wedge...\wedge v_p)=\overline{h}(v_1)\wedge...\wedge \overline{h}(v_p).$$
Then one can obviously  extend the previous definition by linearity.
 \begin{proposition}
 The double form that is associated to the endomorphism $ \widehat{h}$ is $e^h:=1+h+\frac{h^2}{2!}+\frac{h^3}{3!}+...$. In other words we have
 $$ \mathcal{T}(e^h)= \mathcal{T}\left(\sum_{i=0}^\infty\frac{h^p}{p!}\right)=\widehat{h}.$$
Where $h^0=1$ and $h^p=0$ for $p>n$.  In particular, we have $T_{V}\left(\frac{g^p}{p!}\right)={\rm Id}_{\Lambda^p V}.$
 \end{proposition}

\begin{proof}
Let $v_i$ and $w_i$ be arbitrary vectors in $V$ and $1\leq p\leq n$, then
\begin{align*}
\langle  \widehat{h}(v_1\wedge...\wedge v_p),&w_1\wedge...\wedge w_p\rangle=\langle \overline{h}(v_1)\wedge...\wedge \overline{h}(v_p),w_1\wedge...\wedge w_p\rangle\\
&=\frac{1}{p!}\sum_{\sigma\in S_p} \epsilon(\sigma)  \langle \overline{h}(v_{\sigma  (1)})\wedge...\wedge \overline{h} ( v_{\sigma  (p)}),w_1\wedge...\wedge w_p\rangle\\
&= \frac{1}{p!}\sum_{\sigma,\rho \in S_p} \epsilon(\sigma)\epsilon(\rho)  \langle \overline{h}(v_{\sigma  (1)}), w_{\rho(1)}\rangle ...\langle \overline{h}(v_{\sigma  (p)}), w_{\rho(p)}\rangle\\
&= \frac{1}{p!}\sum_{\sigma,\rho \in S_p} \epsilon(\sigma)\epsilon(\rho)  h(v_{\sigma  (1)}, w_{\rho(1)}) ...h(v_{\sigma  (p)}, w_{\rho(p)})\\
&= \frac{h^p}{p!}\big( v_1,...,v_p; w_1,...,w_p\big).
\end{align*}
This completes the proof of the proposition.
\end{proof}

We can now  extend  the exterior algebra endomorphism $ \widehat{h}$  on $\Lambda V$ to an exterior algebra endomorphism on the space ${\mathcal D}(V)$ of double vectors. In the same way as we did for the interior product in the previous paragraph, we can perform this extension  in two natural ways  as follows: \\
We define the right endomorphism
$$\widehat{h}_R: {\mathcal D}(V)\rightarrow {\mathcal D}(V),$$
for a simple double vector $\omega=\theta_1\otimes \theta_2$ by 
$$\widehat{h}_R(\omega)=h_R(\theta_1\otimes \theta_2)=\theta_1\otimes  \widehat{h}(\theta_2).$$
Then one extends the definition using linearity.
Similarly, we define the left extension endomorphism
$$\widehat{h}_L: {\mathcal D}(V)\rightarrow {\mathcal D}(V)$$
by:
$$\widehat{h}_L(\omega)=h_L(\theta_1\otimes \theta_2)= \widehat{h}(\theta_1)\otimes \theta_2.$$

\begin{proposition}\label{algebra endomorphism}
 The endomorphisms $\widehat{h}_L$ and $\widehat{h}_R$   are double exterior algebra endomorphisms.
\end{proposition}
\begin{proof}
\item Without loss of generality, let $\omega=\theta_1\otimes\theta_2$ and $\theta=\theta_3\otimes \theta_4$ be simple double forms then,
\begin{equation*}
\begin{split}
\widehat{h}_R(\omega\theta)&=\widehat{h}_R\big((\theta_1\otimes\theta_2)(\theta_3\otimes\theta_4)\big)=\widehat{h}_R\big(\theta_1\wedge\theta_3\otimes \theta_2\wedge\theta_4\big)\\
&=(\theta_1\wedge\theta_3)\otimes \widehat{h}(\theta_2\wedge\theta_4)=(\theta_1\wedge\theta_3)\otimes(\widehat{h}(\theta_2)\wedge \widehat{h}(\theta_4))\\
&=(\theta_1\otimes\widehat{h}(\theta_2))(\theta_3\otimes \widehat{h}(\theta_4))=\widehat{h}_R(\theta_1\otimes\theta_2)\widehat{h}_R(\theta_3\otimes\theta_4)\\
&=\widehat{h}_R(\omega) \widehat{h}_R(\theta).
\end{split}
\end{equation*}
\end{proof}

\begin{proposition}
Let $\widehat{h}_R,\widehat{h}_L$ be as above and $1\leq p\leq n$ then
\begin{equation}\label{prop5.3}
\widehat{h}_R(\frac{g^p}{p!})=\widehat{h}_L(\frac{g^p}{p!})=\frac{h^p}{p!}.
\end{equation}
Where the metric $g$ is seen here as a $(1,1)$ double exterior vector.
\end{proposition}
\begin{proof}
Let $(e_i)$ be an orthonormal basis for $(V,g)$, then the double vector $g$ splits to  $g=\sum_{i=1}^ne_i\otimes e_i$ and therefore
$$\widehat{h}_R(g)=\widehat{h}_R(\sum_{i=1}^ne_i \otimes e_i)= \sum_{i=1}^ne_i\otimes  \widehat{h}(e_i)    =\sum_{i,j=1}^nh(e_i,e_j)e_i\otimes  e_j   =h.$$
Next,  Proposition (\ref{algebra endomorphism}) shows that
\[ \widehat{h}_R(g^p)=\big(\, \widehat{h}_R(g)\big)^p=h^p.\]
The  proof  for $\widehat{h}_L$ is similar.
\end{proof}

A special case of the previous proposition deserves more attention, namely when $p=n$, we have 
\begin{equation}\label{Phi(g^n)=det(phi)g^n}
\widehat{h}_R(\frac{g^n}{n!})=\widehat{h}_L(\frac{g^n}{n!})=\frac{h^n}{n!}=\det h. \frac{g^n}{n!}.
\end{equation}

The next proposition shows that the endomorphisms $\widehat{h}_R$  and $\widehat{h}_L$  are nothing but the  the right and left multiplication maps in the composition algebra

\begin{proposition}\label{prop1.8} With the above notations we have
\begin{equation}\label{eq24}
\widehat{h}_R(\omega)=e^h\circ \omega,\, {\rm and}\,\,  \widehat{h}_L(\omega)=\omega\circ e^{(h^t)}.
\end{equation}

\end{proposition}
\begin{proof}
 As $ \widehat{h}_R$ is linear in $\omega$, we may assume, without loss of any generality, that the double $(p,q)$ vector $\omega$ is simple that is $\omega=e_{i_1}\wedge ...\wedge e_{i_p}\otimes e_{j_1}\wedge ...\wedge e_{j_q} $. Let us  use multiindex notation and write
 $\omega=e_I\otimes e_J$. From one hand, we have
\begin{align*}
  \widehat{h}_R(\omega)&=e_I\otimes \widehat{h}(e_J)=\sum_{K}\langle \widehat{h}(e_J),e_K\rangle e_I\otimes e_K=\\
&=\sum_{K}  \frac{h^q}{q!} (e_J,e_K)e_I\otimes e_K=\sum_{K,L}  \frac{h^q}{q!} (e_L,e_K)\langle e_L,e_J\rangle e_I\otimes e_K\\
&= \sum_{K,L}  \frac{h^q}{q!} (e_L,e_K)  ( e_L\otimes e_K)\circ (e_I\otimes e_J)  =\frac{h^q}{q!}\circ \omega. 
\end{align*}
To  prove the second assertion we proceed as follows
$$\widehat{h}_L(\omega)=\Big(\widehat{h}_R(\omega^t)\Big)^t=(e^h\circ \omega^t)^t=\omega\circ (e^h)^t=\omega\circ e^{(h^t)}.$$
The fact that $ (e^h)^t=e^{(h^t)}$ results from Proposition \ref{transpose-properties}
 \end{proof}

\begin{corollary}
The adjoint endomorphism of  $  \widehat{h}_R$ (resp. $  \widehat{h}_L$)  is $  \widehat{(h^t)}_R$ (resp. $  \widehat{(h^t)}_L$).

\end{corollary}

 \begin{proof}
 Proposition \ref{transpose-properties} shows that $ (e^h)^t=e^{(h^t)}$ and
\[\langle   \widehat{h}_R(\omega_1),\omega_2\rangle=\langle e^h\circ \omega_1,\omega_2\rangle=\langle \omega_1, e^{(h^t)}\circ \omega_2\rangle=\langle \omega_1,  \widehat{h^t}_R(\omega_2)\rangle.\]
The proof for $  \widehat{h}_L$ is similar.
\end{proof}

Using the facts that both $  \widehat{h}_L$  and $  \widehat{h}_R$ are exterior algebra homomorphisms and the previous corollary  we can easily prove the following technical but  useful identities 

\begin{corollary}
Let $\omega\in {\mathcal D}(V)$ and $h$ be an endomorphism of $V$ then  we have
\begin{equation}\label{iPhi_r}
i_{\omega}\circ   \widehat{h}_R=  \widehat{h}_R \circ  i_{  \scriptstyle \widehat{(h^t)}_R(\omega)   },\,{\rm and}\,\, i_{\omega}\circ   \widehat{h}_L=  \widehat{h}_L \circ  i_{  \scriptstyle \widehat{(h^t)}_L(\omega)   }.
\end{equation}
\end{corollary}
\begin{proof}
Since $  \widehat{h}_R$ is an exterior  algebra endomorphism then for any double vectors $\omega$ and $\theta$ we have
$$\widehat{h}_R(\omega\theta)=\widehat{h}_R(\omega)\widehat{h}_R(\theta)$$
That is,
$$\widehat{h}_R\circ \mu_\omega=\mu_{\widehat{h}_R(\omega)}\circ \widehat{h}_R.$$
Next, take the adjoint of both sides of the previous equation to get
$$i_\omega\circ \widehat{(h^t)}_R=\widehat{(h^t)}_R\circ  i_{\scriptstyle \widehat{h}_R(\omega)}.$$
The proof of the second identity is similar.
\end{proof}

Now we have enough tools to easily prove delicate results of linear algera including the general Laplace expansions of the determinant as follows
\begin{proposition}[Laplace Expansion of the determinant, Proposition 7.2.1 in \cite{Greub-book}]
 For $1\leq p\leq n$, we have
\begin{equation}\label{Laplace-exp}
\frac{(h^t)^{n-p}}{(n-p)!}\circ (\ast\frac{h^p}{p!})=\det h\frac{g^{n-p}}{(n-p)!}\, {\rm and}\,\, \big( \ast \frac{(h^t)^p}{(p)!}\big)\circ   \frac{h^{n-p}}{(n-p)!}=\det h\frac{g^{n-p}}{(n-p)!}.
\end{equation}

\end{proposition}
\begin{proof}
Using the identities (\ref{iPhi_r}) we have
\begin{align*}
\frac{(h^t)^{n-p}}{(n-p)!}\circ(\ast h^p)=& \widehat{(h^t)}_R\circ i_{h^p}\frac{g^n}{n!}=  \widehat{(h^t)}_R\circ i_{ \widehat{h}_R(g^p)}\frac{g^n}{n!}=i_{g^p}\circ  \widehat{(h^t)}_R(\frac{g^n}{n!})\\
&=i_{g^p}\circ  \frac{(h^t)^n}{n!}=\det (h^t) i_{g^p}\circ  \frac{g^n}{n!}=(\det h)\frac{p!}{(n-p)!}g^{n-p}.
\end{align*}
The second identity can be proved in the same way as the first one by using $ \widehat{h}_L$ instead of  $\widehat{h}_R$, or too simply  just by taking the transpose of the first identity.
\end{proof}

To see why the previous identity coincides with the classical Laplace expansion of the determinant we refer  the reader for instance to \cite{Labbialgebraic}.

\begin{proposition}[Proposition 7.2.2, \cite{Greub-book}]
Let $h$ be a bilinear form on the vector space $V$, then
\begin{align*}
i_{*h^p}h^q=& \binom{2n-p-q}{n-p}p!q!(\det h)\frac{h^{p+q-n}}{(p+q-n)!},\\
(*h^p)(*h^q)=& \binom{2n-p-q}{n-p}p!q!(\det h)(*\frac{h^{p+q-n}}{(p+q-n)!}).
\end{align*}
\end{proposition}
\begin{proof}
We use in succession the identities  \ref{prop5.3}, \ref{iPhi_r}, \ref{eq24}, \ref{Laplace-exp}, \ref{interiormetricproduct} and \ref{prop5.3} to get
\begin{align*}
i_{*h^p}h^q &=i_{*h^p} \widehat{h}_R(g^q)= \widehat{h}_R\circ i_{\scriptstyle  \widehat{(h^t)}_R(\ast h^p)}(g^q)\\
&=  \widehat{h}_R\circ  i_{\scriptstyle \frac{(h^t)^{n-p}}{(n-p)!}\circ (\ast h^p)}(g^q)=\frac{p!\det h}{(n-p)!} \widehat{h}_R\circ  i_{g^{n-p}}(g^q)\\
&=\frac{p!q!(2n-p-q)!}{(n-q)!(n-p)!} \widehat{h}_R\big(\frac{g^{p+q-n}}{(p+q-n)!}\big)=\frac{p!q!(2n-p-q)!}{(n-q)!(n-p)!} \frac{h^{p+q-n}}{(p+q-n)!}.
\end{align*}
This proves the first identity. The second one results from the first one by using the identity \ref{star-i-star} as follows
\[     (\ast h^p)(\ast h^q)=   \ast i_{*h^p}\ast(\ast h^q)=     \ast\big(  i_{*h^p}(h^q)\big). \]
\end{proof}

The interior product provides a simple formulation of the Newton (or cofactor) transformations $t_p(h)$ of a bilinear form $h$  and also for its characteristic coefficients $s_k(h)$ \cite{Labbialgebraic} as follows
\begin{proposition}
Let $h$ be a bilinear form on the vector space $V$, then
\begin{enumerate}
\item
 For $0\leq p\leq n$, the $p$-th invariant of $h$ is given by
\[s_p(h):= *\frac{g^{n-p}h^p}{(n-p)!p!}=   i_{\frac{h^p}{p!}}\frac{g^{p}}{p!}.\]

\item For $0\leq p\leq n-1$, the $p$-th Newton transformation of $h$ is given by
\[ t_p(h) :=  *\frac{g^{n-p-1}h^p}{(n-p-1)!p!}=   i_{\frac{h^p}{p!}}\frac{g^{p+1}}{(p+1)!}.\]
\item  More generally, for $0\leq p\leq n-r$,  the $(r,p)$ cofactor transformation \cite{Labbialgebraic} of $h$ is given by 
 
\[ s_{(r,p)}(h):=*\frac{g^{n-p-r}h^p}{(n-p-r)!p!}= i_{\frac{h^p}{p!}}\frac{g^{p+r}}{(p+r)!}.\]
\end{enumerate}
\end{proposition}
%Before embarking in the proof of this proposition, let recall the reader about the definition of the $t_p(h)$ and %$s_p(h)$ given in \cite{Labbialgebraic},
%\begin{equation}
%t_p(h)=*\frac{g^{n-p-1}h^p}{p!(n-p-1)!},
%\end{equation}
%and
%\begin{equation}
%s_p(h)=*\frac{g^{n-p}h^p}{p!(n-p)!}.
%\end{equation}

\begin{proof}
First we use formula (\ref{interiormetricproduct}) to prove (1)  as follows: For $0\leq p\leq n-1$ we have
$$p!t_p(h)=*\frac{g^{n-p-1}h^p}{(n-p-1)!}=i_{\frac{g^{n-p-1}h^p}{(n-p-1)!}}(\frac{g^n}{n!})=i_{h^p}\circ i_{\frac{g^{n-p-1}}{(n-p-1)!}}(\frac{g^n}{n!})=i_{h^p}(\frac{g^{p+1}}{p+1}).$$
 In the same way,  we prove together the relation (2) and its generalization the relation (3) as follows:

$$p!s_{(r,p)}(h)=*\frac{g^{n-p-r}h^p}{(n-p-r)!}=i_{\frac{g^{n-p-r}h^p}{(n-p-r)!}}(\frac{g^n}{n!})=i_{h^p}\circ i_{\frac{g^{n-p-r}}{(n-p-r)!}}(\frac{g^n}{n!})=i_{h^p}(\frac{g^{p+r}}{p+r}).$$

\end{proof}

\begin{remark}
According to \cite{Labbialgebraic},for $0\leq r\leq n-pq$,  the $(r,pq)$ cofactor transformation of a $(p,p)$  double form $\omega$ is defined by
$$h_{(r,pq)}(\omega):=*\frac{g^{n-pq-r}\omega^q}{(n-pq-r)!}.$$
Using the same arguments as above, it is easy to see that
$$h_{(r,pq)}(\omega)=i_{\omega^q}\frac{g^{pq-r}}{(pq-r)!}.$$

\end{remark}

\section{Greub and Greub-Vanstone basic identities}
\subsection{Greub's  Basic identities}
We now state and prove  Greub's basic identities  relating the exterior and composition  products of double forms. 
 %We will give others relation relating these operations.
 \begin{proposition}[Proposition 6.5.1 in \cite{Greub-book}]
 If $h, h_1,...,h_p$ are $(1,1)$-forms, then
 \begin{equation}\label{Greub-Basic-id}
\begin{split}
 i_h(h_1...h_p)&=\sum_{j}\langle h,h_j\rangle  h_1...\hat{h}_j...h_p\\
 &-\sum_{j<k}(h_j\circ h^t\circ h_k+h_k\circ h^t\circ h_j)h_1...\hat{h}_j...\hat{h}_k...h_p.
\end{split}
 \end{equation}
 In particular, if $k=h_1=...=h_p$, we have
 \begin{equation*}
 i_h k^p=p\langle h,k\rangle k^{p-1}-p(p-1)(k\circ {h}^t\circ k)k^{p-2}.
 \end{equation*}

 \end{proposition}
 \begin{proof}
Assume without loss of  generality, that $h=\theta\otimes \vartheta$ and $h_i=\theta_i\otimes \vartheta_i$, where $ \theta, \vartheta, \theta_i,$ and $\vartheta_i$ are in $V^*$,
 %then by definition of the exterior product of double forms, we have
%$$ h_1...h_p=\theta_1\wedge...\wedge\theta_p\otimes \vartheta_1\wedge...\wedge\vartheta_p,$$
then
\begin{align*}
i_h(h_1...h_p)&=i_{(\theta\otimes \vartheta)}(\theta_1\wedge...\wedge\theta_p\otimes \vartheta_1\wedge...\wedge\vartheta_p)\\
&=i_{\theta}\circ \tilde{i}_{\vartheta}(\theta_1\wedge...\wedge\theta_p\otimes \vartheta_1\wedge...\wedge\vartheta_p)\\
&=i_{\theta}\Big( \theta_1\wedge...\wedge\theta_p\Big)\otimes   {i}_{\vartheta}\Big( \vartheta_1\wedge...\wedge\vartheta_p\Big)\\
&=\sum_{j,k}(-1)^{j+k}\langle \vartheta, \vartheta_j\rangle\langle \theta,\theta_k\rangle(\theta_1\wedge...\wedge\hat{\theta_k}\wedge...\wedge\theta_p\otimes \vartheta_1\wedge...\wedge\hat{\vartheta}_j\wedge...\wedge\vartheta_p)
\end{align*}
Where we have used the fact that the ordinary interior product in the exterior algebra $\Lambda(V^*)$ is an antiderivation of degree -1.
Next,  write the previous  sum in three parts for $j=k$, $j<k$ and $j>k$ as follows
\begin{align*}
&i_h(h_1...h_p)=\sum_{j}\langle \vartheta, \vartheta_j\rangle\langle \theta,\theta_j\rangle(\theta_1\wedge...\wedge\hat{\theta_j}\wedge...\wedge\theta_p\otimes \vartheta_1\wedge...\wedge\hat{\vartheta}_j\wedge...\wedge\vartheta_p)\\
&-\sum_{j<k}\langle \vartheta, \vartheta_j\rangle\langle \theta,\theta_k\rangle(\theta_j\otimes \vartheta_k)[\theta_1\wedge...\wedge\hat{\theta_j}\wedge...\wedge\hat{\theta_k}\wedge...\wedge\theta_p]\otimes[ \vartheta_1\wedge...\wedge\hat{\vartheta}_j\wedge...\wedge\hat{\vartheta}_k\wedge...\wedge\vartheta_p]\\
&-\sum_{k<j}\langle \vartheta, \vartheta_j\rangle\langle \theta,\theta_k\rangle(\theta_j\otimes \vartheta_k)[\theta_1\wedge...\wedge\hat{\theta_j}\wedge...\wedge\hat{\theta_k}\wedge...\wedge\theta_p]\otimes[ \vartheta_1\wedge...\wedge\hat{\vartheta}_j\wedge...\wedge\hat{\vartheta}_k\wedge...\wedge\vartheta_p]\\
\end{align*}
Using the defintion of the composition product, one can easily check that
$$ \langle \vartheta, \vartheta_j\rangle\langle \theta,\theta_k\rangle(\theta_j\otimes \vartheta_k)=h_k\circ h^t\circ h_j.$$
Consequently, we can write
\begin{align*}
i_h(h_1...h_p)&=\sum_{j}\langle h,h_j\rangle  h_1...\hat{h}_j...h_p\\
&-\sum_{j<k}(h_k\circ h^t\circ h_j)  h_1...\hat{h}_i...\hat{h}_j...h_p \\
&-\sum_{k<j}(h_k\circ h^t\circ h_j)  h_1...\hat{h}_i...\hat{h}_j...h_p .
\end{align*}
This completes the proof of the  proposition.
 \end{proof}
 \begin{corollary}
 If $h_1...h_p$ are $(1,1)$ double forms, then
 \begin{equation}
 c(h_1...h_p)=\sum_{i}(ch_i)h_1...\hat{h}_i....h_p-\sum_{i<j}(h_j\circ h_i+h_i\circ h_j)h_1...\hat{h}_i...\hat{h}_j...h_p.
 \end{equation}
 In particular, for a $(1,1)$ double form $k$, the contraction of $k^p$  is given by 
 \begin{equation*}
 \cc k^p=p(\cc k) k^{p-1}-p(p-1)(k \circ k)k^{p-2}.
 \end{equation*}
 \end{corollary}
 \begin{proof}
Recall that $h=g$ is a unit  element for the composition product and that the contraction map $\cc$ is the adjoint of the exterior multiplacation map by $g$, The corollary follows immediately from the previous. 
 \end{proof}
As a corollary to the above  Greub's basic identity (\ref{Greub-Basic-id}), Vanstone  proved the following formula  which is in fact  the  main result of  his paper \cite{Vanstone},( formula (27)),
\[ i_{\omega^t}\frac{g^{q+2p}}{(q+2p)!}=(-1)^p\sum_{r}^{}(-1)^r\mu_{\frac{g^{r+q}}{(r+q)!}}\circ i_{\frac{g^r}{r!}}(\omega).\]
Where $\omega$ is any $(p,p)$ double form, and $p,q$ are arbitrary integers.\\
In view of formula (\ref{*iw}) of this paper, the previous identity can be reformulated as follows
\[\ast \frac{g^{n-q-2p}\omega^t}{(n-q-2p)!}=\sum_{r}^{}(-1)^{r+p}\frac{g^{r+q}}{(r+q)!}\frac{c^r}{r!}(\omega).\]
Let $k=n-q-p$, the previous formula then read
\begin{equation}
\ast \Big( \frac{g^{k-p}\omega}{(k-p)!}\Big)=\sum_{r}^{}(-1)^{r+p}\frac{g^{n-p-k+r}}{(n-p-k+r)!}\frac{c^r}{r!}(\omega^t).
\end{equation}
We recover then  formula (15) of \cite{Labbidoubleforms} in Theorem 4.1. Note that  Vanstone's proof of this identity  does not require the $(p,p)$ double form $\omega$ neither to satisfy the first Bianchi identity nor to be a symmetric double form.

\subsection{Greub-Vanstone Basic identities}
Greub-Vanstone basic identities are stated  in the following theorem
\begin{theorem}[\cite{Greub-Vanstone}]\label{Greubformula}
For $1\leq p\leq n$, and for bilinear forms $h_1,...,h_p$ and $k_1,...,k_p$ we have
\[(h_1h_2...h_p)\circ (k_1k_2...k_p)=\sum_{\sigma\in S_p}(h_1\circ k_{\sigma(1)})...(h_p\circ k_{\sigma (p)})=\sum_{\sigma\in S_p}(h_{\sigma (1)}\circ k_{1})...(h_{\sigma (p)}\circ k_{p})\]
In particular, when $h=h_1=...=h_p$ and $k=k_1=...=k_p$ we have the following nice relation:
\begin{equation}
h^p\circ k^p=p!(h\circ k)^p.
\end{equation}
\end{theorem}
%\begin{proof}
%According to proposition (\ref{prop1.8}) and its corollary (\ref{coro 1.9}), we have
%$$h\circ k=\Phi_r(k),$$
%where the endomorphism $\Phi_r$, is induced by $h$. As $\Phi_r$ is an algebra homomorphism, we can write:
%$$\Phi_r(k_1.k_2)=\Phi_r(k_1).\Phi_r(k_2)=(h\circ k_1).(h\circ k_2),$$
%where $\Phi_r=\frac{h^2}{2!}$, by induction, we can easily conclude by remarking that at step $p$ we think
%that  the endomorphism $\Phi_r$ is induced by $\frac{h^p}{p!}.$ (i. e.; $\Phi_r=\frac{h^p}{p!}.$)

%\end{proof}

\begin{proof}
We  assume that $h_i=\theta_i\otimes \vartheta_i$ and $k_i=\theta'_i\otimes \vartheta'_i$, where $ \theta_i, \vartheta_i, \theta'_i,$ and $\vartheta'_i$ are in $V^*$, then by definition of the exterior product of double forms,
 we have
$$ h_1...h_p=\theta_1\wedge...\wedge\theta_p\otimes \vartheta_1\wedge...\wedge\vartheta_p,$$
and
$$k_1...k_p=\theta'_1\wedge...\wedge\theta'_p\otimes \vartheta'_1\wedge...\wedge\vartheta'_p,$$
It follows from the definition of the composition product  that
\[(h_1h_2...h_p)\circ (k_1k_2...k_p)=\det(\langle \theta_i,\vartheta'_j\rangle)[\theta'_1\wedge...\wedge\theta'_p\otimes \vartheta_1\wedge...\wedge\vartheta_p]\]
Now the determinant here can be expanded in two different ways:
\[\det(\langle\theta_i,\vartheta'_j\rangle)=\sum_{\sigma\in S_p}\varepsilon_{\sigma}\langle\theta_1,\vartheta'_{\sigma(1)}\rangle...\langle\theta_p,\vartheta'_{\sigma(p)}\rangle=
\sum_{\sigma\in S_p}\varepsilon_{\sigma}\langle\theta_{\sigma(1)},\vartheta'_1\rangle...\langle\theta_{\sigma(p)},\vartheta'_p\rangle\]
Therefore, using the definition of the composition product, we get:
%\begin{equation*}
\begin{align*}
(h_1h_2...h_p)\circ & (k_1k_2...k_p)=
 \sum_{\sigma\in S_p}\varepsilon_{\sigma}\langle\theta_1,\vartheta'_{\sigma(1)}\rangle...\langle\theta_p,\vartheta'_{\sigma(p)}\rangle
[\theta'_1\wedge...\wedge\theta'_p\otimes \vartheta_1\wedge...\wedge\vartheta_p]\\
=& \sum_{\sigma\in S_p}\langle\theta_1,\vartheta'_{\sigma(1)}\rangle...\langle\theta_p,\vartheta'_{\sigma(p)}\rangle
[\theta'_{\sigma(1)}\wedge...\wedge\theta'_{\sigma (p)}\otimes \vartheta_1\wedge...\wedge\vartheta_p]\\
=& \sum_{\sigma\in S_p} \Big(\prod_{i=1}^p (\theta_i\otimes \vartheta'_i)\circ   ( \theta'_{\sigma(i)}  \otimes  \vartheta'_{\sigma(i)})\Big)
=\sum_{\sigma\in S_p} \Big( \prod_{i=1}^p h_1\circ k_{\sigma(i)}\Big).
\end{align*}
%\end{equation}
If we use the second expansion of  the determinant  we get the second formula using the same arguments.
\end{proof}

%then state some of its important corollaries like
%\[h^p\circ k^p=(h\circ k)^p.\]
%and then end this section with applications to geometry (vanishing of pontryagin numbers..)

\section{Pontrjagin classes and  $p$-Pure curvature tensors }

\subsection{Alternating operator, Bianchi map}
We define the alternating operator as follows:
\begin{align*}
{\rm Alt:}& {\mathcal D}^{p,p}(V^*)\longrightarrow \Lambda^{2p}(V^*)\\
&\omega\mapsto {\rm Alt}(\omega)(v_1,...,v_p,v_{p+1},...v_{2p})\\
&=\frac{1}{(2p)!}\sum_{\sigma\in S_{2p}}\varepsilon(\sigma)\omega(v_{\sigma(1)}\wedge...\wedge v_{\sigma(p)},v_{\sigma(p+1)}\wedge ...\wedge v_{\sigma(2p)}).
\end{align*}

Another basic map in ${\mathcal D}(V^*)$ is the \emph{ first Bianchi map}, denoted by $ \mathfrak{S}$. It maps ${\mathcal D}^{p,q}(V^*)$
into $D^{p+1,q-1}(V^*)$ and is defined as follows. Let $\omega\in {\mathcal D}^{p,q}(V^*)$, set $ \mathfrak{S}\omega=0$ if $q=0$.
Otherwise set
\begin{equation}
 \mathfrak{S}\omega(e_1\wedge...\wedge e_{p+1}, e_{p+2}\wedge...\wedge e_{p+q})=
\frac{1}{p!}\sum_{\sigma\in S_{p+1}}\varepsilon(\sigma)\omega(e_{\sigma(1)}\wedge...\wedge e_{\sigma(p)}, e_{\sigma(p+1)}\wedge e_{p+2}\wedge...\wedge e_{p+q}).
\end{equation}
In other terms, $ \mathfrak{S}$ is a partial alternating operator with respect to the first $(p+1)$ arguments. If we assume that $p=q$, then the composition
\begin{equation}\label{Bp}
 \mathfrak{S}^p:= \mathfrak{S}\circ...\circ  \mathfrak{S},
\end{equation}
is up to a constant factor, the alternating operator $\rm{Alt}$. In particular, we have the following relation observed first by Thorpe \cite{Thorpe} and Stehney \cite{Stehney}
\begin{lemma}
if $\omega\in \ker \mathfrak{S}$, then ${\rm Alt}(\omega)=0.$
\end{lemma}

\begin{lemma}
The linear application ${\rm Alt}$ is surjective.
\end{lemma}
\begin{proof}
If $\omega$ is a $(2p)$-form in $\Lambda^{2p}(V^*)$, then $\omega$ is as well a $(p,p)$ double form whose image under the alternating operator is the $(2p)$-form $\omega$ itself.
\end{proof}

\begin{lemma}
We have the following isomorphism
$${\mathcal D}^{p,p}(V)/\ker {\rm Alt}\cong \Lambda^{2p}(V), $$
In particular, we have the following orthogonal decomposition
$${\mathcal D}^{p,p}(V)=\ker{\rm Alt}\oplus \Lambda^{2p}(V).$$
\end{lemma}

\subsection{p-Pure Riemannian manifolds}
According to  Maillot \cite{Maillot}, a Riemannian $n$-manifold has a \emph{pure curvature tensor} if at each point of the manifold there exists an orthonormal basis $(e_i)$  of the tangent space at this point such that the Riemann curvature tensor $R$ belongs to 
${\rm Span}\{e_i^*\wedge e^*_j\otimes e_i^*\wedge e^*_j:1\leq i<j\leq n \}$. This class contains all conformally flat manifolds, hypersurfaces of space forms and all three dimensional Riemannian manifolds. Maillot proved in \cite{Maillot} that all pontrjagin forms of a pure Riemannian manifold vanish. In this section we are going to refine  this result.

\begin{definition}
Let $1\leq p\leq n/2$ be a positive integer. A Riemannian $n$-manifold  is said to have a \emph{$p$-pure curvature tensor} if at each point of the manifold      there exists an orthonormal basis $(e_i)$  of the tangent space at this point such that  the exterior power $R^p$ of $R$  belongs to
$${\rm Span}\{e_{i_1}^*\wedge ... \wedge e^*_{i_p}\otimes e_{i_1}^*\wedge ... \wedge e^*_{i_p}:1\leq i_1<...<i_p\leq n.\}$$.
\end{definition}

The previous definition can be re-formulated using the exterior product of double forms as follows
\begin{proposition}\label{ppure}
Let $1\leq p\leq n/2$ be a positive integer. A  Riemann $n$-manifold with Riemann curvature tensor $R$ is  $p$-pure if and only if 
at each point of the manifold, there exists a   family $\{h_i:i\in I\}$ of simultaneously diagonalizable symmetric  bilinear  forms on the tangent space   such that  the exterior power $R^p$ of $R$ at that point  belongs to
\[ {\rm Span}\{h_{i_1}...h_{i_p}:i_1,..i_p\in I\}.\]
\end{proposition}

We notice that the condition that the    family $\{h_i:i\in I\}$ consists of simultaneously diagonalizable symmetric  bilinear  forms  is equivalent to the fact that $h_i^t=h_i$ and $h_i\circ h_j=h_j\circ h_i$ for all $i,j\in I$.\\

\begin{proof}
Assume that $R$ is $p$-pure, then by definition we have
\begin{align*}
R^p=&\sum_{1\leq  i_1<...<i_p\leq n }\lambda_{i_1...i_p}   e_{i_1}^*\wedge ... \wedge e^*_{i_p}\otimes e_{i_1}^*\wedge ... \wedge e^*_{i_p}\\
=&\sum_{1\leq  i_1<...<i_p\leq n } \lambda_{i_1...i_p}  \big( e_{i_1}^*\otimes e^*_{i_1}\big) \big( e_{i_2}^*\otimes e^*_{i_2}\big)... \big( e_{i_p}^*\otimes e^*_{i_p}\big)\\
=& \sum_{1\leq  i_1<...<i_p\leq n } \lambda_{i_1...i_p}h_{i_1}...h_{i_p}.
\end{align*}
Where $h_i=e_i^*\otimes e_i^*$. It is clear that $h_i^t=h_i$ and $h_i\circ h_j=\delta_{ij} e_j^*\otimes e_i^*=h_j\circ h_i$.\\
Conversely, assume that there exists a   family $\{h_i:i\in I\}$ of simultaneously diagonalizable symmetric  bilinear  forms  such that   $R^p=\sum_{i_1,...,i_p\in I}\lambda_{i_1...i_p}h_{i_1}...h_{i_p}$. Let $(e_i)$ be an orthonormal basis of the tangent space at the point under consideration that diagonalizes simultaneously all the bilinear forms in the family  $\{h_i:i\in I\}$. Then if $h_{i_k}=\sum_{j_k=1}^{n}\rho_{i_kj_k}e_{j_k}^*\otimes e_{j_k}^*$ for each $k=1,...,p$ we have
 \begin{align*}
 R^p=&\sum_{i_1,...,i_p\in I}\lambda_{i_1...i_p}h_{i_1}...h_{i_p}\\
=& \sum_{i_1,...,i_p\in I}\sum_{j_1,...,j_p=1}^{n}\lambda_{i_1...i_p}\rho_{i_1j_1}...\rho_{i_pj_p}\Big(e_{j_1}^*\otimes e_{j_1}^*\Big)...\Big(e_{j_p}^*\otimes e_{j_p}^*\Big)\\
=& \sum_{i_1,...,i_p\in I}\sum_{j_1,...,j_p=1}^{n}\lambda_{i_1...i_p}\rho_{i_1j_1}...\rho_{i_pj_p}e_{j_1}^*\wedge ... \wedge e_{j_p}^*\otimes e_{j_1}^*\wedge ... \wedge e_{j_p}^*.
\end{align*}
This completes the proof.
\end{proof}

We list below several examples and facts about  $p$-pure manifolds
\begin{enumerate}
\item It is clear that every pure Riemannian  manifold is $p$-pure for any $p\geq 1$. More generally, if a Riemannian manifold is $p$-pure for some $p$ then it is $pq$-pure for any $q\geq 1$.\\
However the converse it is not always true. A Riemannian manifold can be $p$-pure for some $p>1$ without being pure as one can see from the three examples below.
\item A Riemannian manifold of dimension $n=2p$ is always $p$-pure. This follows from the fact that in this case  $R^p$ is proportinal to $g^{n}$.
\item   A Riemannian manifold of dimension $n=2p+1$ is always $p$-pure. In fact, it follows from Proposition 2.1 in \cite{Labbi-eins} that
\[R^p=\omega_1g^{2p-1}+\omega_0g^{2p},\]
where $\omega_1$ is a symmetric bilinear form and $\omega_0$ is a scalar. 
\item A Riemannian manifold with constant $p$-sectional curvature, in the sense of Thorpe \cite{Thorpe},  is $p$-pure. In fact, constant $p$-sectionnal curvature is equivalent to the fact that  $R^p$ is proportional to $g^{2p}$.

\end{enumerate}

We are now ready to state and prove the following Theorem
\begin{theorem}
If a Riemannian $n$-manifold is $k$-pure and $n\geq 4k$  then its Pontrjagin class of degree $4k$ vanishes. 
\end{theorem}

\begin{proof}
Denote by $R$ the Riemann curvature tensor of the given Riemannian manifold. Then the following differential form is a representative of the  Pontrjagin class of degree $4k$ of the manifold  \cite{Stehney} 
\begin{equation}
P_k(R)=\frac{1}{(k!)^2(2\pi)^{2k}}{\rm Alt}\big(R^k\circ R^k\big).
\end{equation}
We are going to show that $P_k(R)$ vanishes.\\
According to proposition \ref{ppure}  there exists a   family $\{h_i:i\in I\}$ of simultaneously diagonalizable symmetric  bilinear  forms  such that   
$$R^k=\sum_{i_1,...,i_k\in I}\lambda_{i_1...i_k}h_{i_1}...h_{i_k}.$$
Therefore, we have
$$R^k\circ R^k=\sum_{i_1,...,i_k\in I\atop j_1,...,j_k\in I}\lambda_{i_1...i_k}\lambda_{j_1...j_k}h_{i_1}...h_{i_k}\circ  h_{j_1}...h_{j_k}.$$

Next, Proposition \ref{Greubformula}, shows that each term of the previous sum is an exterior product of double forms of the form $hi\circ h_j$ each of which is a  symmetric bilinear form  and therefore  belongs to the kernel of the first Bianchi sum $\mathfrak{S}$. On the other hand, the kernel of $\mathfrak{S}$ is closed under exterior products \cite{Kulkarni}, consequently, $R^k\circ R^k$ belongs to the kernel of $\mathfrak{S}$ and therfore ${\rm Alt}\big(R^k\circ R^k\big)=0$.
\end{proof}
\begin{remark}
We remark that the previous theorem can alternatively be proved directly  without using the identity of \ref{ppure} as follows.\\
Let us use multi-index and write $R^k=\sum_I\lambda_Ie_I\otimes e_I$ as in the definition, then
$${\rm Alt}\big(R^k\circ R^k\big)={\rm Alt}\Big(\sum_I\lambda_I^2e_I\otimes e_I\Big)=0.$$
\end{remark}
As a  direct consequence of the previous theorem, we obtain  the following  equivalent version to a result of Stehney  (Th\'{e}or\`{e}me 3.3,  \cite{Stehney}).
\begin{corollary}
Let M be a Riemannian manifold and p an  integer such that, $4p\leq n=\dim M$. If at any point $m\in M$ the Riemann curvature tensor $R$ satisfies
$$R^p=c_pA^p,$$
where $A: T_mM\longrightarrow T_mM$ is symmetric bilinear form and $c_p$ is a constant. Then the differential form ${\rm Alt}(R^p\circ R^p)$ is $0$. 
\end{corollary}
\begin{proof}
Since  $R^p=c_pA...A$ then it is $p$-pure, the result follows from the theorem.

\end{proof}

\address{Mathematics Department\\
College of Science\\
University of Bahrain\\
32038 Bahrain.}
\end{document}